\documentclass[12pt,leqno]{amsart}
\usepackage{amsmath,epsfig,graphicx,color, float}
\usepackage{subcaption}
\usepackage{relsize}
\usepackage{comment}
\usepackage[colorlinks, linkcolor=blue,citecolor=blue]{hyperref}
\usepackage{pdfpages}
\usepackage{graphicx}
\usepackage{mwe}
\usepackage{algorithm}
\usepackage{algorithmicx}
\usepackage{algpseudocode}
\usepackage{natbib}

\usepackage[margin=1.3in]{geometry}

\numberwithin{table}{section}
\numberwithin{figure}{section}
\numberwithin{equation}{section}

\definecolor{darkblue}{rgb}{.2, 0.2,.8}
\definecolor{darkgreen}{rgb}{0,0.5,0.3}
\definecolor{darkred}{rgb}{.8, .1,.1}

\newcommand{\bfT}{\mat{T}}
\newcommand{\bft}{\vect{t}}
\newcommand{\bfe}{\vect{e}}

\newcommand{\bfp}{\vect{\pi}}

\newcommand{\0}{\mat{0}}

\newcommand{\E}{\mathbb{E}}
\renewcommand{\P }{{\mathbb P}}
\newcommand{\ci}{\mathrel{\text{\scalebox{1.07}{$\perp\mkern-10mu\perp$}}}}

\newcommand{\ex}{{\rm e}}

\newtheorem{lemma}{Lemma}[section]
\newtheorem{theorem}[lemma]{Theorem}
\newtheorem{proposition}[lemma]{Proposition}
\newtheorem{definition}[lemma]{Definition}
\newtheorem{corollary}[lemma]{Corollary}

\newtheorem{conjecture}[lemma]{Conjecture}
\newtheorem{remark}{Remark}[section]

\usepackage{bm}
\newcommand{\vect}[1]{\pmb{#1}}
\newcommand{\mat}[1]{\boldsymbol{\bm #1}}

\allowdisplaybreaks

\begin{document}
\title[Multivariate Phase-type distributions for loss modeling]{A tractable class of multivariate phase-type distributions for loss modeling}

\author[M. Bladt]{Martin Bladt}
\address{Department of Mathematical Sciences, University of Copenhagen,
Universitetsparken 5, DK-2100 Copenhagen, Denmark}
\email{martinbladt@math.ku.dk}
\begin{abstract}
Phase-type (PH) distributions are a popular tool for the analysis of univariate risks in numerous actuarial applications. Their multivariate counterparts (MPH$^\ast$), however, have not seen such a proliferation, due to a lack of explicit formulas and complicated estimation procedures.
A simple construction of multivariate phase-type distributions -- mPH -- is proposed for the parametric description of multivariate risks, leading to models of considerable probabilistic flexibility and statistical tractability. The main idea is to start different Markov processes at the same state, and allow them to evolve independently thereafter, leading to dependent absorption times. By dimension augmentation arguments, this construction can be cast under the umbrella of MPH$^\ast$ class, but enjoys explicit formulas which the general specification lacks, including common measures of dependence. Moreover, it is shown that the class is still rich enough to be dense on the set of multivariate risks supported on the positive orthant, and it is the smallest known sub-class to have this property. In particular, the latter result provides a new short proof of the denseness of the MPH$^\ast$ class. In practice, this means that the mPH class allows for the modeling of bivariate risks with any given correlation or copula. We derive an EM algorithm for its statistical estimation and illustrate it on bivariate insurance data. Extensions to more general settings are outlined.

\end{abstract}
\maketitle

\section{Introduction}

The effective description of multivariate random variables is a common challenging task for actuaries and risk managers. There exists an extensive selection of parametric models designed for this task, which very roughly can be categorized into copula and non-copula based. Copula approaches (cf.  \cite{frees1998understanding} and more generally \cite{joe2014dependence}) are easy to implement since often the margins are fitted separately from the dependence structure. However, their dependence structure {can sometimes be stringent}, and in a one-to-one correspondence with celebrated summary statistics for dependence such as rank correlations {(it is worth mentioning that vine copulas have to a large extent overcome this classic drawback, see \cite{joe1997multivariate})}. Moreover, formulas for their density and other standard functionals are not explicit. See also \cite{mikosch2006copulas} for a conceptual critique on the copula methodology. On the other hand, non-copula approaches tend to possess closed-form densities. For instance, \cite{lee2012modeling} (see also  \cite{willmot2015some}) propose infinite mixtures of Erlang vectors for the effective modeling of multivariate risks. A potential drawback of this approach is that, although guaranteed to approximate any possible risk arbitrarily closely, the actual number of Erlang mixture components does not have a natural or physical interpretation beyond a mixing mechanism. { Additionally, Erlang mixtures in the literature assume a common scale parameter, which may cause poor fitting performance when marginals are of different scales.}

This article proposes a parametric model based on the absorption times of Markov processes.  The advantage of this approach, from an interpretation point of view, is that we may consider the size of each realization as the traversing of a multi-state process coming to an end. More specifically, for each distinct risk, we associate a Markov process, all defined on the same state space (without loss of generality). At inception, all risks start at the same state and are thereafter allowed to traverse the state-space independently until their individual absorptions. This is a bold claim since we are implying that dependence is entirely decided before the actual size of the claims is observed. The dependence-generating process is particularly well-suited for interpreting insurance loss severities. For instance, we may think of the initial distribution of the Markov processes as a common event such as an accident or a natural disaster. After the common event, the evolution of the loss severities is then no longer subject to additional interactions. Simple mixing (as in \cite{lee2012modeling}) can be seen as a degenerate case: when the multi-state processes do not jump from any state to any other state. Modeling joint lifetimes as the absorption of the dependent multi-state process is also conceptually aligned with the analysis and development of products within life insurance. { In terms of statistical performance, the scaling issue from mixtures of Erlangs is also alleviated, since there are no constraints on the parameters in the proposed model. Here, the number of components of mixtures is conceptually replaced by matrix dimensions.}

Despite the simple way of generating dependence, we show that the resulting model is versatile and tractable. We show that the model can be embedded in the multivariate phase-type MPH$^\ast$ class of \cite{kulkarni1989new} (after dimension augmentation), and thus enjoys the property of all projections being univariate phase-type (PH) distributed (introduced by \cite{neuts75}; see also \cite{Bladt2017}). We write mPH to distinguish our multivariate model of phase-type distributions from the general class MPH$^\ast$.
However, and in contrast to the latter class, we derive closed-form formulas for its density, cumulative distribution function, moments, and the most common measures of dependence. We provide proof of the denseness of the mPH class on the positive orthant, along the lines of \cite{johnson1988denseness}, and derive a fully explicit EM algorithm for their Maximum-Likelihood (ML) estimation, generalizing the formulas of \cite{asmussen1996fitting} to the multi-dimensional case.

Multivariate phase-type distributions were first treated in \cite{assaf1984multivariate}, which, together with the construction of \cite{kulkarni1989new}, are the most commonly used classes. The main advantage of our proposed model with respect to the previous contributions is the ability to estimate the model (and its right-censored and fractional extensions) to data in a very straightforward manner. 
The recent contribution \cite{furman2021} considered scaling of PH distributions which resulted in multivariate models with background risk. The statistical and tail behaviour of related models was considered in \cite{albrecher2021continuous}. The mPH class differs from the latter conceptually in that the dependence does not occur as a common shock to all the loss severities, but rather as a shared component at the inception of the claim-size processes.

The rest of the article is structured as follows. Section \ref{sec:mPH} introduces the mPH class and derives its basic properties and measures of dependence. Subsequently, Section \ref{sec:dens} is devoted to showing a deeper property of mPH distributions: their denseness on the positive orthant. The statistical estimation is developed in terms of an EM algorithm in Section \ref{sec:est}, and an illustration is provided in Section \ref{sec:illustration}. Relevant inhomogeneous and fractional extensions are outlined in Section \ref{extensions}, and finally Section \ref{sec:conclusion} concludes.


\section{Multivariate Phase-type distributions}\label{sec:mPH}

\subsection{Notation}
In what follows, we use the following conventions. The vector inequality $\{X\le x\}$ is understood as $\bigcap_{i=1}^d \{X_i\le x_i\},$ and similarly for related notation. For a scalar $\lambda$ and a vector $x$, we understand by $\lambda x$ as the vector $(\lambda x_i)_{i=1,\dots,d}$, and similarly for division. The operators $\otimes$ and $\oplus$ denote the Kronecker multiplication and addition between matrices, respectively.

For any square matrix $\mat{A}$ of dimension $p\times p$, we define
\begin{align*}
	u( \mat{A})=\dfrac{1}{2 \pi i} \oint_{\Gamma}u(w) (w \mat{I} -\mat{A} )^{-1}dw \,,\quad \mat{A}\in \mathbb{R}^{p\times p} \,,
\end{align*}
with $\Gamma$ a simple path enclosing the eigenvalues of $\mat{A}$, and $\mat{I}$ the identity matrix of the same dimension.

\subsection{Definition}
Let $ ( J_t^{(k)} )_{t \geq 0}$, $k=1,\dots,d$, be homogeneous Mar\-kov pure-jump processes on the common state space $\{1, \dots, p, p+1\}$, with states $1,\dots,p$ transient and $p+1$ absorbing, so defining
\begin{align*}
p^{(k)}_{jl}(s,t)=\P(J_t=l|J_s=j)\,,\quad 0\le j,l\le p+1,\:\:0\le k\le d \,,
\end{align*}
we may write
$$\mat{P}_k(s,t)=\exp(\mat{\Lambda}_k(t-s))=
\left( \begin{array}{cc}
		\exp(\bfT_k(t-s)) &  \bfe-\exp(\bfT_k(t-s)) \bfe \\
		\0 & 1
	\end{array} \right)\in\mathbb{R}^{(p+1)\times(p+1)},$$
for $s<t,$ $0\le k\le d$, where $\mat{\Lambda}_k(t)$ are intensity matrices.

For future reference, we write $\bfe_k$ for the $k$-th canonical basis vector in $\mathbb{R}^p$, $\bfe=\sum_{i=1}^p \bfe_i$, and
$$\bfT_k=(t_{ij}^{(k)})_{i,j=1,\dots,p}\,,\quad \bft_k=-\bfT_k\bfe=(t^{(k)}_1,\dots,t^{(k)}_p)^{\mathsf{T}}\,,\quad \lambda^{(k)}_i=-t_{ii}>0\,,\:\:\: i=1,\dots,p\,.$$

The process $\{J_t^{(k)}\}_{t\ge0}$ is assumed to generate the $k$-th risk, $0\le k\le d$, by traversing through different states before reaching absorption. Their dependence structure is completely specified as follows:
\begin{align}\label{dependence_def}
J_0^{(k)}=J_0^{(l)},\quad  ( J_t^{(k)} )_{t \geq 0} {\ci}_{J_0^{(1)}} ( J_t^{(l)} )_{t \geq 0,\:l \neq k},\quad  k,l\in\{1,\dots,d\},
\end{align}
which means that they all start at the same state, but otherwise evolve independently thereafter (note that the $t=0$ case is trivially true). For simplicity, we denote $J_0=J_0^{(1)}$, and $\P(J_0=j)=\pi_j$, $j=1,\dots,p$, and $\bfp=(\pi_1,\dots,\pi_p)$.

Thus, the following random variables
\begin{align}\label{components_def}
	X_i = \inf \{ t >  0 : J^{(i)}_t = p+1 \}\,, \quad i=1,\dots,d,
\end{align}
are all univariate phase-type distributed. Their joint behaviour will be our primary focus in the sequel.

\begin{definition}[mPH class]
We say that a random vector $X\in\mathbb{R}_+^d$ has a multivariate phase-type distribution if each component variable $X_i,\:\:i=1,\dots,d$ is given as in \eqref{components_def}, that is, they are the absorption times of distinct Markov jump processes having the structure \eqref{dependence_def}, which amounts to the following defining properties:
\begin{enumerate}
\item An initial state is drawn, and all $d$ processes start at such a state.
\item Thereafter, each of the $d$ Markov processes evolves independently, until their respective absorptions.
\end{enumerate}
\noindent Moreover, we use the notation $$X\sim \mbox{mPH}(\bfp,\mathcal{T}), \quad \mbox{where}\quad \mathcal{T}=\{\bfT_1,\dots,\bfT_d\}.$$ Without loss of generality, we may assume that $\bfp$ has dimension $p$ and the $\bfT_k$, $k=1,\dots,d$, have dimension $p\times p$.
\end{definition}
\begin{remark}\rm
{ Throughout the remainder of the article, we will use matrices ${\bfT}_k$ of no particular substructure, so that the total degrees of freedom (number of free parameters) is given by $p-1+dp^2$.}
\end{remark}

We begin by showing that this definition falls into a very general class of multivariate distributions with phase-type marginals, the MPH$^\ast$ class of \cite{kulkarni1989new}. 
\begin{proposition}[mPH $\subset$ MPH$^\ast$]
The mPH class is contained in the MPH$^\ast$ class.
\end{proposition}
\begin{proof}
{\color{black}
Define the parameters
\begin{align*}
\tilde\bfp_k=(\pi_k\bfe_k^\mathsf{T},\vect{0},\dots,\vect{0})\in\mathbb{R}^{p d}, \quad k=1,\dots,p,
\end{align*}

\begin{align*}
\tilde{\bfT}_k=
\left(\begin{matrix}{}
  \bfT_1 &\bft \bfe_k^\mathsf{T} &  &\\
  &  \bfT_2 &\bft \bfe_k^\mathsf{T}  &\\
  & & \ddots & \\
  & & & & \bfT_d
\end{matrix}\right)\in\mathbb{R}^{(pd)\times(pd)},\quad k=1,\dots,p,
\end{align*}
\begin{align*}
{\mat{R}}=
\left(\begin{matrix}{}
  \bfe&\vect{0} & \cdots &\\
  \vect{0} &  \bfe &\vect{0}  &\\
  \vdots& & \ddots & \\
  & & & & \bfe
\end{matrix}\right)\in\mathbb{R}^{(pd)\times d}.
\end{align*}
Then with
\begin{align*}
&\tilde\bfp =(\tilde\bfp_1,\dots,\tilde\bfp_p)\in\mathbb{R}^{p^2d}\\
& \tilde{\bfT}=
\left(\begin{matrix}{}
  \tilde{\bfT}_1 & &  &\\
  &  \tilde{\bfT}_2 &  &\\
  & & \ddots & \\
  & & & & \tilde{\bfT}_d
\end{matrix}\right)\in\mathbb{R}^{(p^2d)\times(p^2d)}
,\quad \tilde{\mat{R}}=
\left(\begin{matrix}{}
  \mat{R}\\
  \mat{R} \\
  \vdots \\
 \mat{R}
\end{matrix}\right)\in \mathbb{R}^{(p^2d)\times d},
\end{align*}
it is easy to see that the reward-collecting mechanism in this higher-dimensional space leads to precisely the mPH class.}
\end{proof}
Notice that the dimension of the latter representation is much larger than the individual marginal dimension of the mPH class. This implies that fitting methods for MPH$^\ast$ are prohibitively slow when applied to the mPH subclass, which motivates us to derive a novel estimation procedure in a later section, specifically designed for our case. However, belonging to the MPH$^\ast$ class does imply following desirable property:
\begin{corollary}[Projections]
Let $X$ be a multivariate phase-type random vector. Then any inner product $$\langle X,u \rangle=\sum_{i=1}^d X_i u_i,\quad u>0.$$ is univariate phase-type distributed. 
\end{corollary}

\subsection{Basic properties}
We proceed to derive closed-form formulas for several functionals of interest for the mPH class. Not only are these properties desirable from a multivariate distribution perspective, but they will be crucial in later sections for deriving closed-form formulas for measures of dependence and the EM algorithm.
We define the cumulative distribution function as $$F_X(x)=\P(X_1\le x_1,X_2\le x_2,\dots,X_d\le x_d),\quad x\in\mathbb{R}_+^d,$$ and the survival function as $$S_X(x)=\P(X_1>x_1,X_2> x_2,\dots,X_d> x_d), \quad x\in\mathbb{R}_+^d.$$ Clearly $F_X(x)\neq 1-S_X(x)$ for $d>1$, but they can be related by inclusion-exclusion formulas.

\begin{proposition}[Cumulative distribution function]
Let $X\sim \mbox{mPH}(\bfp,\mathcal{T})$ be a multivariate phase-type random vector. Then
\begin{align*}
F_X(x)=\sum_{j=1}^p\pi_j \prod_{i=1}^d(1-\bfe_j^\mathsf{T}\exp(\bfT_i x_i)\bfe),\quad x\in\mathbb{R}_+^d.
\end{align*}
For the survival function, we may express it in terms of the survival function of a univariate phase-type distribution of appropriate size, as follows.
\begin{align*}
S_X(x)=\P(Y(x)>1),\quad x\in\mathbb{R}_+^d,
\end{align*}
for $Y(x)\sim \mbox{PH}(\tilde \bfp, \bfT(x))$, with
\begin{align}\label{big_pi}
\tilde\bfp=(\pi_1 (\bfe_1^\mathsf{T}\otimes\cdots\otimes \bfe_1^\mathsf{T}),\dots,\pi_p (\bfe_p^\mathsf{T}\otimes\cdots\otimes \bfe_p^\mathsf{T})),
\end{align}
and
\begin{align}\label{big_t}
	\bfT(x)=\left(\begin{matrix}{}
  \bfT_1 x_1 \oplus\cdots\oplus \bfT_d x_d& &  &\\
  & \bfT_1 x_1 \oplus\cdots\oplus \bfT_d x_d &  &\\
  & & \ddots & \\
  & & & & \bfT_1 x_1 \oplus\cdots\oplus \bfT_d x_d
\end{matrix}\right),
\end{align}
where the diagonal block matrix is repeated $p$ times.
\end{proposition}
\begin{proof}
We have that
\begin{align*}
&\P(X_1\le x_1,X_2\le x_2,\dots,X_d\le x_d)\\
&=\sum_{j=1}^p \P(X_1\le x_1,X_2\le x_2,\dots,X_d\le x_d| J_0=j)\P(J_0=j)\\
&=\sum_{j=1}^p\pi_j \prod_{i=1}^d\P(X_i\le x_i| J_0=j)=\sum_{j=1}^p\pi_j \prod_{i=1}^d(1-\bfe_j^\mathsf{T}\exp(\bfT_i x_i)\bfe),
\end{align*}
and similarly for the survival function
\begin{align*}
&\P(X_1> x_1,X_2> x_2,\dots,X_d>x_d)\\
&=\sum_{j=1}^p \P(X_1> x_1,X_2> x_2,\dots,X_d> x_d| J_0=j)\P(J_0=j)\\
&=\sum_{j=1}^p\pi_j \prod_{i=1}^d\P(X_i> x_i| J_0=j)=\sum_{j=1}^p\pi_j \prod_{i=1}^d\bfe_j^\mathsf{T}\exp(\bfT_i x_i)\bfe\\
&=\sum_{j=1}^p\pi_j (\bfe_j^\mathsf{T}\otimes\cdots\otimes \bfe_j^\mathsf{T})\exp(\bfT_1 x_1 \oplus\cdots\oplus \bfT_d x_d)\bfe\\
&=\sum_{j=1}^p\pi_j (\bfe_j^\mathsf{T}\otimes\cdots\otimes \bfe_j^\mathsf{T})\exp(\bfT_1 x_1 \oplus\cdots\oplus \bfT_d x_d)\bfe,
\end{align*}
which implies the second claim.
\end{proof}
\begin{corollary}[Density]
Let $X\sim \mbox{mPH}(\bfp,\mathcal{T})$ be a multivariate phase-type random vector. Then
\begin{align*}
f_X(x)=\sum_{j=1}^p\pi_j \prod_{i=1}^d\bfe_j^\mathsf{T}\exp(\bfT_i x_i)\bft_i =f_{Y(x)}(1),\quad x\in\mathbb{R}_+^d,
\end{align*}
where $Y(x)$ is univariate phase-type with parameters given in \eqref{big_pi} and \eqref{big_t}.
\end{corollary}
\begin{proof}
Follows by repeated partial differentiation.
\end{proof}
The quantiles, even for $d=1$, are in general not explicitly available, and thus numerical methods are required to calculate, for instance, the value-at-risk, tail-value-at-risk, conditional tail expectation, or expected shortfall.

\subsection{Moments and measures of dependence}
Moments are of interest for any multivariate distribution since the provide a glimpse into the dependence between different elements of the random vector.
\begin{theorem}[Laplace transform and moments]
Let $X\sim \mbox{mPH}(\bfp,\mathcal{T})$ be a multivariate phase-type random vector. Then its joint Laplace transform is given by
\begin{align*}
\E(\exp(-uX))=\sum_{j=1}^p\pi_j\prod_{i=1}^d\bfe_j^{\mathsf{T}}(u_i\mat{I}-\bfT_i)^{-1}\bft_i,\quad u\in\mathbb{R}_+^d.
\end{align*}
For $\theta_1,\dots,\theta_d>-1$, we have that
\begin{align*}
\E(X_1^{\theta_1}\cdots X_d^{\theta_d})=\sum_{j=1}^p \pi_j\prod_{i=1}^d \Gamma(\theta_i+1)\bfe_j^{\mathsf{T}}(-\bfT_i)^{-\theta_i}\bfe.
\end{align*}
\end{theorem}
\begin{proof}
We have that
\begin{align*}
\E(\exp(-uX))&=\sum_{j=1}^p\E(\exp(-uX)|J_0=j)\P(J_0=j)\\
&=\sum_{j=1}^p\pi_j\prod_{i=1}^d\E(\exp(-u_iX_i)|J_0=j)=\sum_{j=1}^p\pi_j\prod_{i=1}^d\bfe_j^{\mathsf{T}}(u_i\mat{I}-\bfT_i)^{-1}\bft_i.
\end{align*}
A similar argument holds for the moments.
\end{proof}
We may then easily calculate, for instance, the Pearson correlation between any pair of marginals.
For such result, we first denote by
$$\sigma_k=\left(2\sum_{j=1}^p\pi_j\bfe_j^{\mathsf{T}}(-\bfT_k)^{-2}\bfe -(\sum_{j=1}^p\pi_j\bfe_j^{\mathsf{T}}(-\bfT_k)^{-1}\bfe)^2\right)^{1/2},\quad k=1,\dots,d,$$ the standard deviation of each marginal. 
\begin{corollary}[Pearson correlation]
Let $X\sim \mbox{mPH}(\bfp,\mathcal{T})$ be multivariate phase-type random vector. Then the pairwise Pearson correlations are given by
\begin{align*}
&\rho_{X_k,X_l}\\
&=\frac{\sum_{j=1}^p\pi_j\bfe_j^{\mathsf{T}}(-\bfT_k)^{-1}\bfe \,\bfe_j^{\mathsf{T}}(-\bfT_l)^{-1}\bfe-(\sum_{j=1}^p\pi_j\bfe_j^{\mathsf{T}}(-\bfT_k)^{-1}\bfe)(\sum_{j=1}^p\pi_j\bfe_j^{\mathsf{T}}(-\bfT_l)^{-1}\bfe)}{\sigma_k\sigma_l
},
\end{align*}
with $k\neq l$.
\end{corollary}

Copula-based methods are particularly interested in measures of dependence which are invariant under marginal transformations, and in particular, rank-based measures have gained popularity. We now proceed to show that for the mPH class, such rank-based dependence measures are also explicit. 

Recall that for  $X=(X_1,X_2)$ and $Z=(Z_1,Z_2)$  i.i.d. random vectors, Kendall's correlation at the population level is given by
\begin{align*}
\P((X_1-Z_1)(X_2-Z_2)>0)-\P((X_1-Z_1)(X_2-Z_2)<0).
\end{align*}

\begin{theorem}[Kendall correlation]
Let $X\sim \mbox{mPH}(\bfp,\mathcal{T})$ be multivariate phase-type random vector. Then the pairwise Kendall correlations are given for $k,l\in\{1,\dots,d\}$ by
\begin{align*}
&\tau_{X_k,X_l}\\
&=4\sum_{i=1}^p\sum_{j=1}^p\pi_i\pi_j   (\bfe_i^\mathsf{T}\otimes \bfe_j^\mathsf{T}) [-\bfT_k  \oplus \bfT_k ]^{-1} (\bfe\otimes\bft_k)
   (\bfe_i^\mathsf{T}\otimes \bfe_j^\mathsf{T})[-\bfT_l  \oplus \bfT_l ]^{-1}(\bfe\otimes\bft_l)-1.
\end{align*}
\end{theorem}
\begin{proof}
It is standard that
\begin{align*}
\tau_{X_k,X_l}&=4 \int_{0}^{\infty} \int_{0}^{\infty} S_{(X_k,X_l)}(x, y) f(x, y) d x d y-1,\end{align*}
and so
\begin{align*}
&(\tau_{X_k,X_l}+1)/4\\
&=\int_{0}^{\infty} \int_{0}^{\infty} 
\sum_{i=1}^p\pi_i \bfe_i^\mathsf{T}\exp(\bfT_k x)\bfe\, \bfe_i^\mathsf{T}\exp(\bfT_l y)\bfe
 \sum_{j=1}^p\pi_j \bfe_j^\mathsf{T}\exp(\bfT_k x)\bft_k \bfe_j^\mathsf{T}\exp(\bfT_l y)\bft_l  
 d x d y\\
 &=\sum_{i=1}^p\sum_{j=1}^p\int_{0}^{\infty} \int_{0}^{\infty} 
\pi_i \bfe_i^\mathsf{T}\exp(\bfT_k x)\bfe\, \bfe_i^\mathsf{T}\exp(\bfT_l y)\bfe
 \pi_j \bfe_j^\mathsf{T}\exp(\bfT_k x)\bft_k \bfe_j^\mathsf{T}\exp(\bfT_l y)\bft_l  
 d x d y\\
  &=\sum_{i=1}^p\sum_{j=1}^p\pi_i\pi_j \int_{0}^{\infty}  \bfe_i^\mathsf{T}\exp(\bfT_k x)\bfe\,  \bfe_j^\mathsf{T}\exp(\bfT_k x)\bft_k dx
  \int_{0}^{\infty} 
   \bfe_i^\mathsf{T}\exp(\bfT_l y)\bfe
 \bfe_j^\mathsf{T}\exp(\bfT_l y)\bft_l  
 d y\\
  &=\sum_{i=1}^p\sum_{j=1}^p\pi_i\pi_j \int_{0}^{\infty}  (\bfe_i^\mathsf{T}\otimes \bfe_j^\mathsf{T})\exp(\bfT_k x \oplus \bfT_k x)(\bfe\otimes\bft_k) dx\\
  &\quad\quad\quad\quad\quad\quad \times \int_{0}^{\infty}(\bfe_i^\mathsf{T}\otimes \bfe_j^\mathsf{T})\exp(\bfT_l y \oplus \bfT_l y)(\bfe\otimes\bft_l)  
 d y\\
   &=\sum_{i=1}^p\sum_{j=1}^p\pi_i\pi_j   (\bfe_i^\mathsf{T}\otimes \bfe_j^\mathsf{T}) [-\bfT_k  \oplus \bfT_k ]^{-1} (\bfe\otimes\bft_k)
   (\bfe_i^\mathsf{T}\otimes \bfe_j^\mathsf{T})[-\bfT_l  \oplus \bfT_l ]^{-1}(\bfe\otimes\bft_l).  
\end{align*}
\end{proof}

Similarly, for  $X=(X_1,X_2)$, $(Y_1,Y_2)$ and $Z=(Z_1,Z_2)$  i.i.d. random vectors, Spearman's correlation at the population level is given by
\begin{align*}
\P((X_1-Z_1)(X_2-Y_2)>0)-\P((X_1-Y_1)(X_2-Z_2)<0).
\end{align*}

\begin{theorem}[Spearman correlation]
Let $X\sim \mbox{mPH}(\bfp,\mathcal{T})$ be multivariate phase-type random vector. Then the pairwise Spearman correlations are given for $k,l\in\{1,\dots,d\}$ by
\begin{align*}
&\rho^S_{X_k,X_l}\\
&=12\sum_{j=1}^p\pi_j
\left(1+(\bfp\otimes \bfe_j^{\mathsf{T}})[\bfT_k \oplus \bfT_k]^{-1}(\bfe\otimes \bft_k)\right)
 \left(1+(\bfp\otimes \bfe_j^{\mathsf{T}})[\bfT_l \oplus \bfT_l]^{-1}(\bfe\otimes \bft_l)\right)-3.
\end{align*}
\end{theorem}
\begin{proof}
It is standard that
\begin{align*}
\tau_{X_k,X_l}&=12 \int_{0}^{\infty} \int_{0}^{\infty} F_{X_k}(x)F_{X_l}(y) f(x, y) d x d y-3,\end{align*}
and so
\begin{align*}
&(\rho^S_{X_k,X_l}+3)/12\\
&=\int_{0}^{\infty} \int_{0}^{\infty} 
(1-\bfp\exp(\bfT_k x)\bfe)\, (1-
\bfp \exp(\bfT_l y)\bfe)\,
 \sum_{j=1}^p\pi_j \bfe_j^\mathsf{T}\exp(\bfT_k x)\bft_k \bfe_j^\mathsf{T}\exp(\bfT_l y)\bft_l  
 d x d y\\
 &=\sum_{j=1}^p\pi_j\int_{0}^{\infty}(1-\bfp\exp(\bfT_k x)\bfe) \bfe_j^\mathsf{T}\exp(\bfT_k x)\bft_k dx\, \int_{0}^{\infty} 
 (1-
\bfp \exp(\bfT_l y)\bfe)\,
  \bfe_j^\mathsf{T}\exp(\bfT_l y)\bft_l  
 d y\\
 &=\sum_{j=1}^p\pi_j
\left(1-\int_0^\infty(\bfp\otimes \bfe_j^{\mathsf{T}})\exp(\bfT_k x\oplus \bft_kx)(\bfe\otimes \bft_k)dx\right)\\
&\quad\quad\quad\quad\times
 \left(1-\int_0^\infty(\bfp\otimes \bfe_j^{\mathsf{T}})\exp(\bfT_l y\oplus \bft_ly)(\bfe\otimes \bft_l)dy\right)\\
 &=\sum_{j=1}^p\pi_j
\left(1-(\bfp\otimes \bfe_j^{\mathsf{T}})[-\bfT_k \oplus \bfT_k]^{-1}(\bfe\otimes \bft_k)\right)
 \left(1-(\bfp\otimes \bfe_j^{\mathsf{T}})[-\bfT_l \oplus \bfT_l]^{-1}(\bfe\otimes \bft_l)\right). 
\end{align*}
\end{proof}

\begin{remark}\rm
{
Although, in principle, by their denseness property, any correlation is achievable for a large enough dimension $p$, it is worth noticing that for a fixed $p$, the entire range of correlation may not be captured. For instance, for $p=1$, we have two independent exponential random variables, so the correlation can only be zero.
}
\end{remark}

\begin{remark}[Copulas]\rm
The mPH class in not based on copulas. Nonetheless, it may give rise to very different shapes of copula densities with a relatively small state space. The copula densities, as defined for a bivariate random vector by
\begin{align*}
c_X(u,v)=\frac{f_X(x_1,x_2)}{f_{X_1}(x_1)f_{X_2}(x_2)},\quad u=F_{X_1}(x_1),\:v=F_{X_2}(x_2)\:\in [0,1]
\end{align*}
are in general not explicit, but can be very effectively calculated. 

Consider for instance $p=3$, and the following sojourn intensities: $a=5$, $b=20$ and $c=140$. We define six mPH models according to the following parameters:
\begin{align*}
\bfp=\bfe /3,\quad 
\bfT_1=\left(\begin{matrix}{}
  -a&1& 1 \\
  1 &  -b &1 \\
  1& 1& -c  \\
\end{matrix}\right)
\end{align*}
\begin{align*}
 \bfT_2=\left(\begin{matrix}{}
  -\alpha^{(1)}&1& 1 \\
  1 &  -\alpha^{(2)} &1 \\
  1& 1& -\alpha^{(3)}  \\
\end{matrix}\right), \quad (\alpha^{(1)},\alpha^{(2)},\alpha^{(3)})\in \mathcal{P}((a,b,c)),
\end{align*}
where $\mathcal{P}((a,b,c))$ is the set of all six permutations of the vector $(a,b,c)$. It follows immediately that all six mPH models have the same marginals, and thus the copulas, depicted by their contours in Figure \ref{copulas}, are a standardized measure of dependence when permuting the state space of the second marginal.

\begin{figure}[!htbp]
\centering
\includegraphics[width=0.8\textwidth]{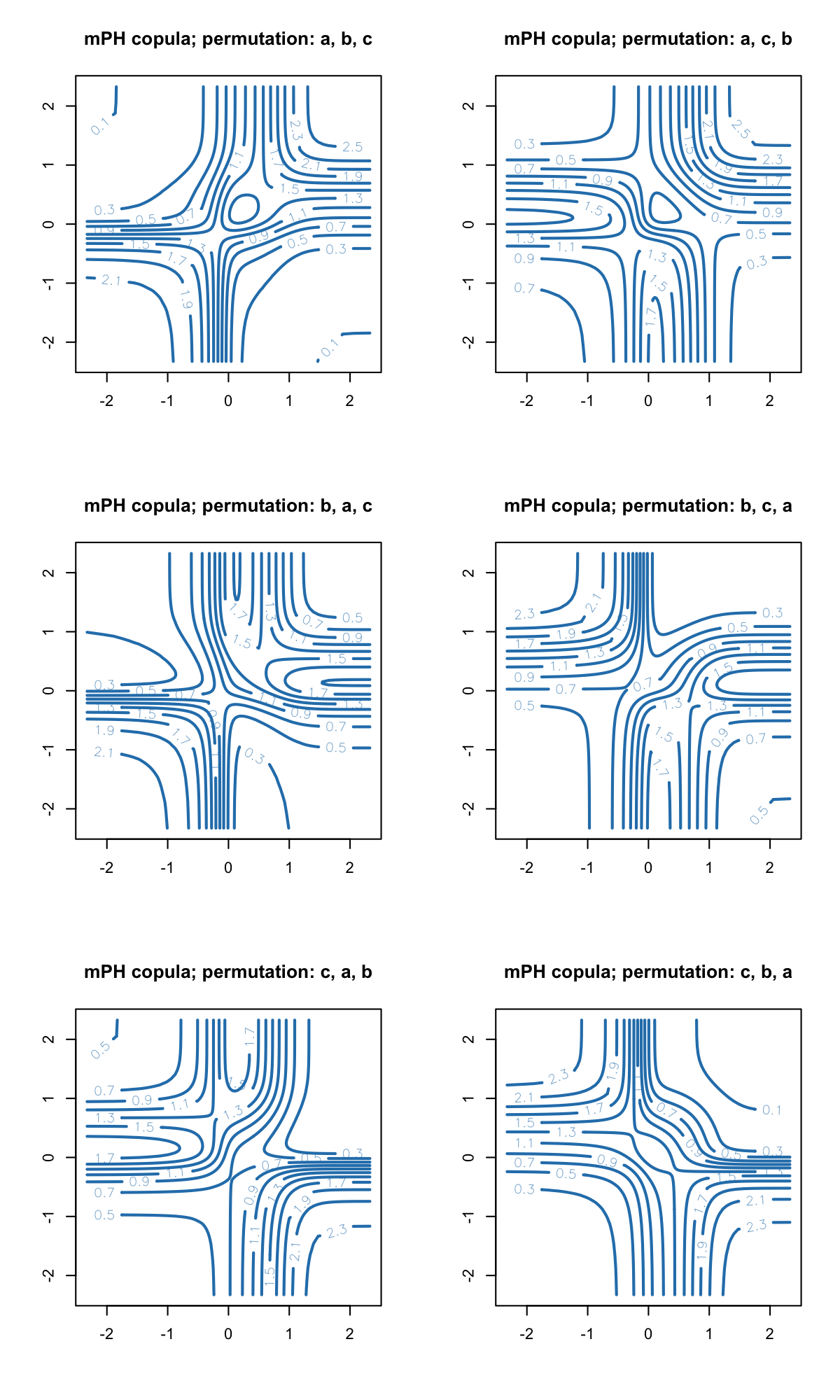}
\caption{{ Copulas associated with the density functions of a bivariate mPH distribution based on permuting the state-space of the second marginal.  Standard Gaussian margins are used for visualization purposes.}
} \label{copulas}
\end{figure}

\end{remark}

\section{Denseness on the positive orthant}\label{sec:dens}
The denseness of infinite mixtures of Erlang vectors was established in \cite{lee2012modeling}. We show that the same holds true for finite mixtures, as a way to establish the property for general finite-order mPH distributions. Although a substantial part of the proof could be taken from the latter reference, we proceed using an alternative and elementary method of proof, inspired by the univariate case given in \cite{johnson1988denseness}{, also based on Erlang substructures}. This method of proof also sheds light into one possible constructive way of interpreting additional mixture components -- or more generally, additional elements of the state-space -- with respect to their contribution towards approximating a target distribution.

Recall that weak convergence is characterised by point-wise convergence of the multivariate cumulative distribution functions at continuity points of the limit.

Define for any $\lambda>0$, 
$$d_F(\lambda,k)=\P(X\in C(\lambda,k)), \quad k\in\mathbb{N}^d,$$
where $C(\lambda,k)=\{x|\, x_i\in((k_i-1)/\lambda,k_i/\lambda], \,i=1,\dots,d\}$. Further, let 
$$E_k(u;\lambda)=1-\sum_{l=1}^{k-1}\frac{1}{l!}(\lambda u)^l\exp(-\lambda u)=1-\sum_{l=1}^{k-1}p_l(u),\quad u>0,$$
i.e. the univariate Erlang cumulative distribution function, and with $p_l(u)$, $l=1,2,\dots$, being Poisson probabilities with mean $\lambda u$.

\begin{theorem}[Denseness of infinite mixtures]
Let $F$ be a distribution concentrated on $\mathbb{R}_+^d$. Define for $n\in\mathbb{N}$	,
\begin{align*}
F_n(x)= \sum_{k_1=1}^\infty\cdots\sum_{k_d=1}^\infty d_F(n,(k_1,\dots,k_d))E_{k_1}(x_1;n)\dots E_{k_d}(x_d;n).
\end{align*}
Then for each $n\in\mathbb{N}$, $F_n$ is a
distribution function and
$$F_n(x)\to F(x),\quad n\to\infty $$
at every point of continuity $x$ of $F$.
\end{theorem}
\begin{proof}
The sum of all $d_F(n,(k_1,\dots,k_d))\ge 0$ is unity, since they measure probabilities at sets which generate a partition. Thus $F_n(x)$ is a multivariate distribution function. Now let $x$ be a continuity point of $F$. Now,
\begin{align*}
&F_n(x)\\
&=\sum_{k_1=1}^\infty\cdots\sum_{k_d=1}^\infty d_F(n,(k_1,\dots,k_d))E_{k_1}(x_1;n)\dots E_{k_d}(x_d;n)\\
&=\sum_{k_1=1}^\infty\cdots\sum_{k_{d-1}=1}^\infty E_{k_1}(x_1;n)\dots E_{k_{d-1}}(x_d;n) \sum_{k_d=1}^\infty d_F(n,(k_1,\dots,k_d)) (1-\sum_{i=1}^{k_d-1}p_i(x_d)).
\end{align*}
Notice that 
\begin{align*}
&\sum_{k_d=1}^\infty d_F(n,(k_1,\dots,k_d)) (1-\sum_{i=1}^{k_d-1}p_i(x_d))\\
&=\sum_{k_d=1}^\infty d_F(n,(k_1,\dots,k_d)) -\sum_{k_d=1}^\infty d_F(n,(k_1,\dots,k_d))\sum_{i=1}^{k_d-1}p_i(x_d)\\
&=\sum_{k_d=1}^\infty d_F(n,(k_1,\dots,k_d)) -\sum_{i=0}^{\infty}p_i(x_d)\sum_{k_d=i+1}^\infty d_F(n,(k_1,\dots,k_d))\\
&=\sum_{i=0}^{\infty}p_i(x_d)\sum_{k_d=1}^{i} d_F(n,(k_1,\dots,k_d)).
\end{align*}
Thus, applying the last relation $d$ times, we obtain
\begin{align*}
F_n(x)&=\sum_{i_1=1}^\infty\cdots\sum_{i_d=1}^\infty p_{i_1}(x_1)\dots p_{i_d}(x_d)
\sum_{k_1=1}^{i_1}\dots\sum_{k_d=1}^{i_d}d_F(n,(k_1,\dots,k_d))\\
&=\sum_{i_1=1}^\infty\cdots\sum_{i_d=1}^\infty p_{i_1}(x_1)\dots p_{i_d}(x_d)
F(i).\\
\end{align*}
Now define the independent random variables $Y_{n,i}\sim\mbox{Poisson}(n\cdot x_i)$, and $X_{n,i}=Y_{n,i}/n$. From the above expression it is follows that 
$$\E(F((X_{n,1},\dots,X_{n,d})))=F_n(x).$$ Since, almost surely, $$\lim_{n\to \infty}X_{n,i}= x_i,\quad i=1,\dots,d,$$
then dominated convergence implies that $$F_{n}(x)\to \E(F(x))= F(x),\quad n\to\infty,$$ as desired.
\end{proof}

We immediately obtain the following result.

\begin{corollary}
Let $X$ be a random vector in $\mathbb{R}_+^d$. Then there exist a sequence of random vectors $\{X_n\}_{n\in\mathbb{N}}$ in $\mathbb{R}_+^d$, where each $X_n$ has a distribution that is an infinite mixture of $d$-dimesional random vectors with independent Erlang components, such that
$$X_n\stackrel{d}{\to}X,\quad n\to\infty.$$
\end{corollary}

We now allow for finite mixtures as well.

\begin{theorem}[Approximation by finite mixtures]
Let $F$ be a distribution concentrated on $\mathbb{R}_+^d$. Define for $n\in\mathbb{N}$ and	{$m=(m_1,\dots,m_d)\in\mathbb{N}^d$},
\begin{align*}
F_{m,n}(x)=\frac{1}{F(m/n)} \sum_{k_1=1}^{m_1}\cdots\sum_{k_d=1}^{m_d} d_F(n,(k_1,\dots,k_d))E_{k_1}(x_1;n)\dots E_{k_d}(x_d;n).
\end{align*}
Then $$\lim_{m\to\infty}\sup_{x\in\mathbb{R}_+^d}|F_n(x)-F_{m,n}(x)| =0.$$
\end{theorem}
\begin{proof}
{ First note that for $X\sim F$,
\begin{align*}
&\sum_{k_1=m_1+1}^\infty\cdots\sum_{k_d=m_d+1}^\infty d_F(n,(k_1,\dots,k_d))\\
&=\P(X_1>m_1/n,\dots, X_d>m_d/n)\\
&\le 1-\P(X_1\le m_1/n,\dots, X_d \le m_d/n)=1-F(m/n).
\end{align*}
Then we have that
}
\begin{align*}
&|F_n(x)-F_{m,n}(x)|\\
&\le\left|\sum_{k_1=m_1+1}^\infty\cdots\sum_{k_d=m_d+1}^\infty d_F(n,(k_1,\dots,k_d))E_{k_1}(x_1;n)\dots E_{k_d}(x_d;n)\right|\\
&+\left|
\left(1-\frac{1}{F(m/n)}\right) \sum_{k_1=1}^{m_1}\cdots\sum_{k_d=1}^{m_d} d_F(n,(k_1,\dots,k_d))E_{k_1}(x_1;n)\dots E_{k_d}(x_d;n) \right|\\
&\le\left|\sum_{k_1=m_1+1}^\infty\cdots\sum_{k_d=m_d+1}^\infty d_F(n,(k_1,\dots,k_d))\right|\\
&\quad\quad+\left|
\left(1-\frac{1}{F(m/n)}\right) \sum_{k_1=1}^{m_1}\cdots\sum_{k_d=1}^{m_d} d_F(n,(k_1,\dots,k_d)) \right|\\
&=\left|\sum_{k_1=m_1+1}^\infty\cdots\sum_{k_d=m_d+1}^\infty d_F(n,(k_1,\dots,k_d))\right|+\left|
1-F(m/n)\right|\\
&\le2\left|
1-F(m/n)\right|,
\end{align*}
which is independent in $x$, and goes to zero as all entries of $m$ go to infinity.
\end{proof}

In particular we obtain the following limit.
\begin{corollary}
Let $F$ be a distribution concentrated on $\mathbb{R}_+^d$. Then
\begin{align*}
\lim_{n\to\infty}\lim_{m\to\infty}F_{m,n}(x)=F(x),\quad \forall x\in\mathbb{R}^d.
\end{align*}
\end{corollary}
{The above now yields the following main result, noting that the class of mPH distributions is closed under mixtures, and that multivariate Erlang distributions with independent marginals fall trivially into the mPH class.}

\begin{theorem}[Denseness of the mPH class]
Let $X$ be a random vector in $\mathbb{R}_+^d$. Then there exist a sequence of random vectors $\{X_n\}_{n\in\mathbb{N}}$ in $\mathbb{R}_+^d$, where each $X_n$ has a multivariate phase-type distribution, such that
$$X_n\stackrel{d}{\to}X,\quad n\to\infty.$$
In particular, the $X_n$ may be taken as finite mixtures of $d$-dimensional random vectors with independent Erlang components.
\end{theorem}

{ It should be mentioned that although the mPH class is dense on the class of positively-supported distributions, for convergence to be achieved, we require the number of components to be allowed to diverge. That is, for a fixed and finite dimension, denseness does not hold.}

Another dense subclass are Markovian Arrival Processes (MAP's). However, the mPH class is contained in the latter and thus is the smallest (with respect to set inclusion) known dense subclass of the MPH$^\ast$ distributions.

\begin{conjecture}
The mPH class is the smallest possible subclass of the MPH$^\ast$ distributions with general PH margins which is dense on the set of distributions supported on the positive orthant.
\end{conjecture}

\section{Estimation}\label{sec:est}
We now derive a fully-explicit EM algorithm for the ML estimation of the mPH class. In contrast to the estimation methods for the MPH$^\ast$ class, the formulas are not based on projections onto univariate domains, but instead we directly compute the conditional expectations on the multivariate domain.

Let $F\sim\mbox{mPH}( \bfp , \mathcal{T} ).$ Assume we have a sample $x^{(1)},\dots,x^{(n)}\sim F$ with associated latent sample paths $\{J_t^{(i,m)}\}_{t\ge0}$, $i=1,\dots,d$, $m=1,\dots,n$. We make the following definitions. 
\begin{align*}
B_k&=\sum_{i=1}^d\sum_{m=1}^n 1\{J_0^{(i,m)}=k\},\quad k=1,\dots,p,\\
N^{(i)}_{ks}&=\sum_{m=1}^n \sum_{t\ge0}1\{J_{t-}^{(i,m)}=k,J_{t}^{(i,m)}=s\},\quad k,s=1,\dots,p,\:\: i=1,\dots,d,\\
N^{(i)}_{k}&=\sum_{m=1}^n \sum_{t\ge0}1\{J_{t-}^{(i,m)}=k,J_{t}^{(i,m)}=p+1\},\quad k=1,\dots,p,\:\: i=1,\dots,d,\\
Z_k^{(i)}&=\sum_{m=1}^n \int_0^\infty1\{J_t^{(i,m)} =k\}dt, \quad k=1,\dots,p,\:\: i=1,\dots,d.
\end{align*}
These statistics are not observable, but help with constructing an effective EM algorithm. Notice also their interpretation in terms of the dynamics of the underlying Markov process. Denote by $\Xi$ (respectively $\xi$) the observed information of the sample at the population level (respectively, observation level).
Then, the completely observed likelihood can be written as follows:
\begin{equation}
\mathcal{L}_c( \bfp , \mathcal{T};\xi)=
\prod_{k=1}^{p} {\pi_k}^{B_k} \prod_{i=1}^d\prod_{k=1}^{p}\prod_{s\neq k} {t^{(i)}_{ks}}^{N^{(i)}_{ks}}e^{-t^{(i)}_{ks}Z^{(i)}_k}\prod_{k=1}^{p}{t^{(i)}_k}^{N^{(i)}_k}e^{-t^{(i)}_{k}Z^{(i)}_k},   
\end{equation}

which is seen to conveniently fall into the exponential family of distributions, and thus has explicit maximum likelihood estimators.

The E-step consists on calculating the expected value of the log likelihood, given the sample, while the M-step maximises the likelihood given the conditional expectations in place of the actual statistics. The latter step is straightforward and thus the details are omitted. We thus only derive the former, which depends on the joint distribution. We get by Bayes' formula
\begin{align*}
 \mathbb{E}(B_k\mid \Xi=\xi)
&=\sum_{i=1}^d\sum_{m=1}^n \mathbb{E}(1\{J_0^{(i,m)}=k\}\mid \Xi=\xi)\\
&=d\sum_{m=1}^n \P(J_0^{(m)}=k\mid \Xi=\xi)\\
&=d\sum_{m=1}^n \frac{\P( J_0^{(m)}=k)\P(X_1\in dx_1^{(m)},\dots, X_d\in dx_d^{(m)}\mid J_0^{(m)}=k)}{\P(X_1\in dx_1^{(m)},\dots, X_d\in dx_d^{(m)})}\\
&=d\sum_{m=1}^n \frac{\pi_k \prod_{i=1}^d {\bfe_k}^{ \mathsf{T}}\exp( \bfT_i x^{(m)}_i) \bft_i }{\sum_{j=1}^p \pi_j \prod_{i=1}^d {\bfe_k}^{ \mathsf{T}}\exp( \bfT_i x^{(m)}_i) \bft_i},
\end{align*}
\begin{align*}
&\mathbb{E}(Z_k^{(i)}\mid \Xi=\xi)\\
&=\sum_{m=1}^n\E(\int_0^{x_i^{(m)}}1\{J_t^{(i,m)} =k\}dt\mid \Xi=\xi)\\
&=\sum_{m=1}^n\int_0^{x_i^{(m)}}\P(J_t^{(i,m)} =k\mid \Xi=\xi)dt\\
&=\sum_{m=1}^n\int_0^{x_i^{(m)}}\frac{\P( J_t^{(i,m)} =k)\P(X_1\in dx_1^{(m)},\dots, X_d\in dx_d^{(m)}\mid J_t^{(i,m)} =k)}{\P(X_1\in dx_1^{(m)},\dots, X_d\in dx_d^{(m)})}dt\\
&=\sum_{m=1}^n\int_0^{x_i^{(m)}}\frac{
\sum_{j=1}^p \pi_j 
{\bfe_k}^{\mathsf{T}}
\exp(\bfT_i(x^{(m)}_i-t))\bft_i
\prod_{l\neq i}{\bfe_j}^{ \mathsf{T}}\exp( \bfT_l x^{(m)}_l) \bft_l
{\bfe_j}^{\mathsf{T}}\exp(\bfT_i t)\bfe_k
 }{\sum_{j=1}^p \pi_j \prod_{i=1}^d {\bfe_j}^{ \mathsf{T}}\exp( \bfT_i x^{(m)}_i) \bft_i}dt\\
 &=\sum_{m=1}^n \frac{\sum_{j=1}^p\pi_j
 \prod_{l\neq i}{\bfe_j}^{ \mathsf{T}}\exp( \bfT_l x^{(m)}_l) \bft_l
 }{\sum_{j=1}^p \pi_j \prod_{i=1}^d {\bfe_j}^{ \mathsf{T}}\exp( \bfT_i x^{(m)}_i) \bft_i}
 \int_0^{x_i^{(m)}} {\bfe_k}^{\mathsf{T}}
\exp(\bfT_i(x^{(m)}_i-t))\bft_i
{\bfe_j}^{\mathsf{T}}\exp(\bfT_i t)\bfe_k
dt.
\end{align*}
Further,
{\tiny
\begin{align*}
&\mathbb{E}(N_{ks}^{(i)}\mid \Xi=\xi)\\
&=\sum_{m=1}^n \frac{\E(\sum_{t\ge0}1\{J_{t-}^{(i,m)}=k,J_{t}^{(i,m)}=s\},X_1\in dx_1^{(m)},\dots, X_d\in dx_d^{(m)})}{\P(X_1\in dx_1^{(m)},\dots, X_d\in dx_d^{(m)})}\\
&=\sum_{m=1}^n\sum_{j=1}^p\sum_{l=1}^p \pi_k\frac{\E(\sum_{t\ge0}1\{J_{t-}^{(i,m)}=k,J_{t}^{(i,m)}=s\}1\{X_1\in dx_1^{(m)},\dots, X_d\in dx_d^{(m)}\} 1\{J^{(i,m)}_{x_i^{(m)}-}=l\}t^{(i)}_l\mid J_0^{(m)}=j)}{\sum_{j=1}^p \pi_j \prod_{i=1}^d {\bfe_j}^{ \mathsf{T}}\exp( \bfT_i x^{(m)}_i) \bft_i}\\
&=
\sum_{m=1}^n \frac{\sum_{j=1}^p\pi_j
 \prod_{l\neq i}{\bfe_j}^{ \mathsf{T}}\exp( \bfT_l x^{(m)}_l) \bft_l
 }{\sum_{j=1}^p \pi_j \prod_{i=1}^d {\bfe_j}^{ \mathsf{T}}\exp( \bfT_i x^{(m)}_i) \bft_i}
\sum_{l=1}^p 
\E(\sum_{t\ge0}1\{J_{t-}^{(i,m)}=k,J_{t}^{(i,m)}=s\}1\{X_i\in dx_i^{(m)}\} 1\{J^{(i,m)}_{x_i^{(m)}-}=j\}\mid J_0^{(m)}=j)t^{(i)}_l\\
&=
\sum_{m=1}^n \frac{\sum_{j=1}^p\pi_j
 \prod_{l\neq i}{\bfe_j}^{ \mathsf{T}}\exp( \bfT_l x^{(m)}_l) \bft_l
 }{\sum_{j=1}^p \pi_j \prod_{i=1}^d {\bfe_j}^{ \mathsf{T}}\exp( \bfT_i x^{(m)}_i) \bft_i}
\sum_{l=1}^p 
t_{ks}^{(i)} \int_0^{x_i^{(m)}} \bfe_s^{\mathsf{T}}\exp(\bfT_i(x_i^{(m)}-t))\bfe_l \bfe_j^{\mathsf{T}}\exp(\bfT_i t)\bfe_k dt
\,t^{(i)}_l,
\end{align*}
}
so that summing up, we obtain
\begin{align*}
&\mathbb{E}(N_{ks}^{(i)}\mid \Xi=\xi)\\
&=
\sum_{m=1}^n \frac{\sum_{j=1}^p\pi_j
 \prod_{l\neq i}{\bfe_j}^{ \mathsf{T}}\exp( \bfT_l x^{(m)}_l) \bft_l
 }{\sum_{j=1}^p \pi_j \prod_{i=1}^d {\bfe_j}^{ \mathsf{T}}\exp( \bfT_i x^{(m)}_i) \bft_i} 
t_{ks}^{(i)}\bfe_s^{\mathsf{T}} \left(\int_0^{x_i^{(m)}}\exp(\bfT_i(x_i^{(m)}-t))\bft_i \bfe_j^{\mathsf{T}}\exp(\bfT_i t) dt\right)\bfe_k.
\end{align*}
Finally,
\begin{align*}
\mathbb{E}(N_{k}^{(i)}\mid \Xi=\xi)
&=\sum_{m=1}^n \E(\sum_{t\ge0}1\{J_{t-}^{(i,m)}=k,J_{t}^{(i,m)}=p+1\}\mid \Xi=\xi)\\
&=\sum_{m=1}^n \P(J^{(i)}_{x^{(m)}_i -}=k\mid \Xi=\xi)\\
&=\sum_{m=1}^n \frac{\P(X_1\in dx_1^{(m)},\dots, X_d\in dx_d^{(m)} \mid J^{(i,m)}_{x^{(m)}_i -}=k)\P(J^{(i,m)}_{x^{(m)}_i -}=k)}{\P(X_1\in dx_1^{(m)},\dots, X_d\in dx_d^{(m)})}\\
&=\sum_{m=1}^n t_k^{(i)}\frac{\P(X_1\in dx_1^{(m)},\dots, X_d\in dx_d^{(m)} \mid J^{(i,m)}_{x^{(m)}_i -}=k)}{\sum_{j=1}^p \pi_j \prod_{i=1}^d {\bfe_j}^{ \mathsf{T}}\exp( \bfT_i x^{(m)}_i) \bft_i}\\
&=\sum_{m=1}^n\sum_{j=1}^p\pi_j t_k^{(i)}\frac{
{\bfe_j}^{ \mathsf{T}}\exp( \bfT_i x^{(m)}_i) \bfe_k
\prod_{l\neq i}{\bfe_j}^{ \mathsf{T}}\exp( \bfT_l x^{(m)}_l) \bft_l 
}{\sum_{j=1}^p \pi_j \prod_{i=1}^d {\bfe_j}^{ \mathsf{T}}\exp( \bfT_i x^{(m)}_i) \bft_i}.
\end{align*}
It is not hard to verify that for $d=1$, all these formulas indeed reduce to the corresponding univariate phase-type variants, found in \cite{asmussen1996fitting}.


For simplicity, we provide the full procedure in Algorithm \ref{alg;mult_ph}, where we avoid redundant matrix exponential evaluations -- which are computationally costly. Notice that integrals of matrix exponentials do not require numerical evaluation, but can be circumvented using the results of \cite{van1978computing}, which in particular gives the identity
	\begin{align*}
		\exp \left( \left(  \begin{array}{cc}
		\bfT & \bft \, \bfp \\
		\0 & \bfT
	\end{array}  \right) y \right)=  \left(  \begin{array}{cc}
		\ex^{\bfT y } & \int_{0}^{y} \ex^{ \bfT (y-u)} \bft \bfp \ex^{ \bfT u}du \\
		\0 & \ex^{\bfT y } 
	\end{array}  \right)\,.
	\end{align*}
	This implies that integrals of the form  $$\int_{0}^{y} \ex^{ \bfT (y-u)} \bft \bfp \ex^{ \bfT u}du,$$ may be obtained as a sub-matrix of a single matrix exponential evaluation.

{\small
\begin{algorithm}[]
\caption{EM algorithm for mPH distributions}\label{alg;mult_ph}
\begin{algorithmic}
\State \textit{\textbf{Input}: a sample of size $n$, $x=({x^{(1)}}^{\mathsf{T}},\dots, {x^{(n)}}^{\mathsf{T}})$, where each $x^{(m)}\in\mathbb{R}_+^d.$}\\
\begin{enumerate} 
\item[ 1)]\textit{Matrix exponentials:} Compute for each sample point (and thus their super-index is omitted in this step):
\begin{align*}
a_{kij}&=\bfe_k^{\mathsf{T}}\exp(\bfT_i x_i)\bfe_j, \quad k,j=1,\dots,p,\:\: i=1,\dots,d,\\
a_{ki}&=\sum_{j} t^{(i)}_j a_{kij}, \quad k=1,\dots,p,\:\: i=1,\dots,d,\\
a_{k,-i}&=\prod_{j\neq i} a_{kj},\quad k=1,\dots,p,\:\:i=1,\dots,d,\\
a_k&=\prod_i a_{ki},\quad k=1,\dots,p,\\
\tilde a_{k,i} &=\sum_j\pi_j a_{j,-i}a_{jik},\quad k=1,\dots,p,\:\: i=1,\dots, d,\\
a&=\sum_k \pi_k a_k\\
b_{skij}&={\bfe_s}^{\mathsf{T}}\int_0^{x_i}
\exp(\bfT_i(x_i-t))\bft_i
{\bfe_j}^{\mathsf{T}}\exp(\bfT_i t)
dt\, \bfe_k,\quad s,k,j=1,\dots ,p,\:\: i=1,\dots , d,\\
b_{ski}&=\sum_j \pi_j (a_{j,-i}b_{skij}),\quad s,k=1,\dots ,p,\:\: i=1,\dots , d,
\end{align*}

\item[ 2)]\textit{E-step:} compute the conditional expectations
\begin{align*}
     \mathbb{E}(B_k\mid \Xi=\xi)&=d\pi_k \sum^n(a_k/a), \quad k=1,\dots,p\\
     \mathbb{E}(Z_k^{(i)}\mid \Xi=\xi)&=  \sum^n( b_{kki}/a), \quad k=1,\dots,p,\:\:i =1,\dots,d,\\
     \mathbb{E}(N_{ks}^{(i)}\mid \Xi=\xi)&=t^{(i)}_{ks}\sum^n( b_{ski}/a),\quad \quad s,k=1,\dots ,p,\:\: i=1,\dots , d,\\
     \mathbb{E}(N_{k}^{(i)}\mid \Xi=\xi)&=t_k^{(i)}\sum^n(\tilde a_{ki}/a),\quad k=1,\dots ,p,\:\: i=1,\dots , d
\end{align*}

\item[3)] \textit{M-step:} let 
\begin{align*}
\hat \pi_k&=\frac{\mathbb{E}(B_k\mid \Xi=\xi)}{d\cdot n},\quad k=1,\dots,p,\\
\widehat{t_{ks}^{(i)}}&=\frac{\mathbb{E}(N^{(i)}_{ks}\mid\Xi=\xi)}{\mathbb{E}(Z^{(i)}_{k}\mid \Xi=\xi)}, \quad s,k=1,\dots,p,\:\: i=1,\dots,d.\\
\widehat{t_{k}^{(i)}}&=\frac{\mathbb{E}(N^{(i)}_{k}\mid \Xi=\xi)}{\mathbb{E}(Z^{(i)}_{k}\mid \Xi=\xi)}, \quad k=1,\dots,p,\:\: i=1,\dots,d.\\
\widehat{t_{kk}^{(i)}}&=-\sum_{s\neq k} \widehat {t^{(i)}_{ks}}-\widehat{t_k^{(i)}}, \quad k=1,\dots,p,\:\: i=1,\dots,d.
\end{align*}

\item[4)] Update the current parameters to $({\bfp},\mathcal{T}) =(\hat{{\bfp}},\hat{\mathcal{T}})$. Return to step 1 unless a stopping rule is satisfied.
\end{enumerate}
    \State \textit{\textbf{Output}: fitted representation $(\hat{{\bfp}},\hat{\mathcal{T}})$.}
\end{algorithmic}
\end{algorithm}
}

%

Standard arguments for the EM algorithm for exponential families yields the following result, which despite not guaranteeing convergenge to a global optimum, it does guarantee monotonicity and thus eventual convergence.

\begin{proposition}[Convergence]
The likelihood function is increasing at each iteration of the EM algorithm for mPH distributions. In particular, convergence is guaranteed to a possibly local maximum.
\end{proposition}

\begin{remark}\rm
{
Using the \texttt{fit()} function for \texttt{mph} objects from the \texttt{matrixdist} package\footnote{\texttt{R} package openly available at \url{https://github.com/martinbladt/matrixdist_1.0}.}, most datasets will be correctly estimated within one thousand iterations (absolute tolerance of $10^{-6}$), which for a dataset of e.g. $10^4$ datapoints with $d=2$ will take just a few minutes (depending on computer specifications). Initial parameters are by default randomly generated when initializing an \texttt{mph} object.
}
\end{remark}

\section{An insurance illustration}\label{sec:illustration}
This section illustrates the statistical methodology developed in the previous section on real-life insurance data. The principal aim is to show the feasibility and accuracy of the algorithm. That said, we do not intend to do a systematic comparison against all relevant multivariate models. A full practical guide to how to effectively choose the order $p$ and structure of the underlying processes, as well as their corresponding simplified EM algorithms, is the subject of additional research, which is currently under preparation (where the useful interval-censored case is also studied). { Preliminary results suggest that for models with a unique mode, a small $p$ (smaller than $5$) is often sufficient, with no special structure, while Coxian structures and large $p$ (larger than $10$) perform better when multimodal marginals are present.} Presently, $p$ is chosen by trial and error, and we allow for the most general phase-type structure possible (no zeros in $\bfT$).

The dataset is the Loss-ALAE dataset, consisting of $n=1500$ bivariate observations, the first margin being an insurance loss, and the second one the corresponding allocated loss adjustment
expense. The dataset was considered in \cite{frees1998understanding}, see also Section 7.4 of
\cite{joe2014dependence}.
There are $34$ loss observations are which are right-censored, which we will presently consider as fully observed. Deriving interval-censored versions of the EM algorithm above is the subject of further research.
The loss variable ranges from $10$ to $2.2$MM, with quartiles of
$4,000$, $12,000$ and $35,000$, while the ALAE variable ranges from $15$ to $0.5$MM with quartiles of $2,300$, $5,500$ and $12,600$. Following \cite{joe2014dependence}, we divide all data points by $10,000$, for easier numerical implementation.

{Phase-type marginals are popular for statistical modeling due to their closed-form formulas and denseness properties and are sometimes linked through a copula to obtain dependence. However, by the results of previous sections, we may fit a tractable multivariate phase-type distribution directly.} Thus, we consider fitting a $4$-dimensional mPH via the EM algorithm of the previous section, and for reference, we also consider separately fitting $4$-dimensional univariate PH distributions to each marginal and modeling the dependence structure with a copula (effectively, the IFM method). The order $p=4$ was found by fitting various orders of increasing size until no significant improvement of the likelihood function occurred, and the initial parameters were randomly generated. The results are given in Table \ref{table:loss_alae_fits}. We observe that the performance of the dependence structure implied from the mPH model is advantageous compared to the copulas, given that PH margins are chosen. 
Further evidence of the adequate fit is provided in Figures \ref{scatter_contour} and \ref{histograms}.
Naturally, the IFM method also allows for non-PH margins, but this would not make the dependence structure directly comparable. For extensions of the tail behaviour of the margins, we refer to Section \ref{extensions}. The parameters  are given by
\begin{align*}
\bfp&=(0.408\,,0.441\,,0.135\,,0.016),\\ 
\bfT_1&=\left(\begin{matrix}{}
  -0.381&0.336& 0 & 0 \\
  0 &  -1.797 &0 & 0.005\\
  0.007& 0.014& -0.077 & 0 \\
    0.024& 0& 0 & -0.025\\
\end{matrix}\right),\\ 
\bfT_2&=\left(\begin{matrix}{}
 -1.481&  0.9&  0.043& 0\\
0 &-2.526 & 0.017&  0.004\\
 0.236 & 0.025& -0.417&  0\\
0& 0&  0.085& -0.085\\
\end{matrix}\right).
\end{align*}
{We observe that all states are possible initial states. For the first marginal, the following transitions are possible: $1\to2\to4\to1$, $3\to1$, $3\to2$, which essentially means that the third state is not accessed very often (only if chosen as initial state). Given that the third state has the second largest mean sojourn time, with average $1/0.077\approx 13$, the state is acting like a possible additive shock. Of course, the interpretation is in terms of latent variables, and without further covariates, cannot be linked to a physical process. For the second marginal, similar considerations hold, although there are more jump possibilities. When interested in the joint behaviour one may inspect each possible path. For example, notice that when starting in state four, the marginals behave very similarly, both having a large initial sojourn time and then immediately jumping to a state of comparable and smaller sojourn time, so that a strong positive relation is created for large values.
}

\begin{figure}[!htbp]
\centering
\includegraphics[width=\textwidth]{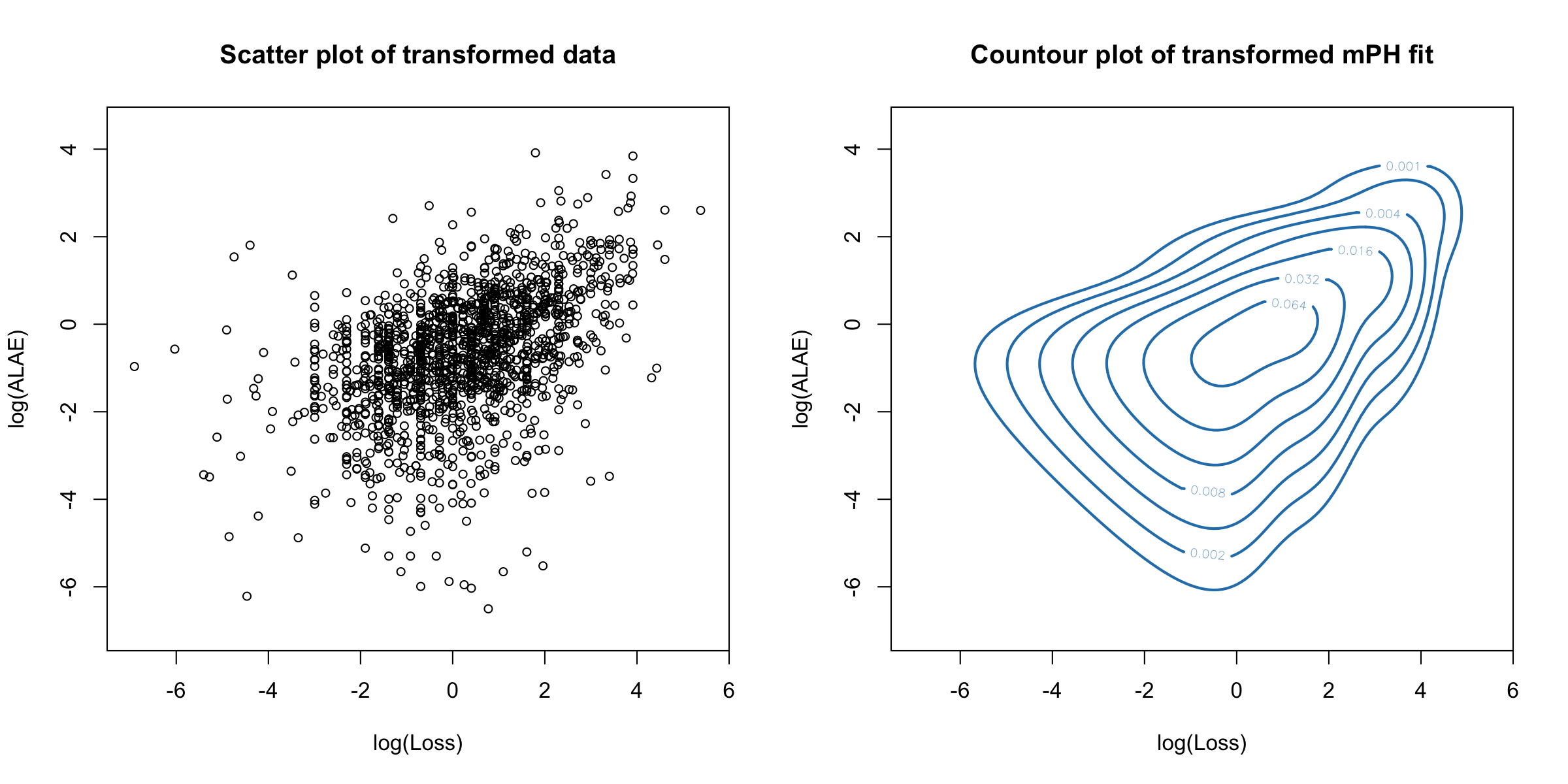}
\caption{Scatter plot of the Loss-ALAE data, together with contour lines of the fitted mPH density. Both are log-transformed in both marginals (after estimation) for visualization purposes.
} \label{scatter_contour}
\end{figure}

\begin{figure}[!htbp]
\centering
\includegraphics[width=\textwidth]{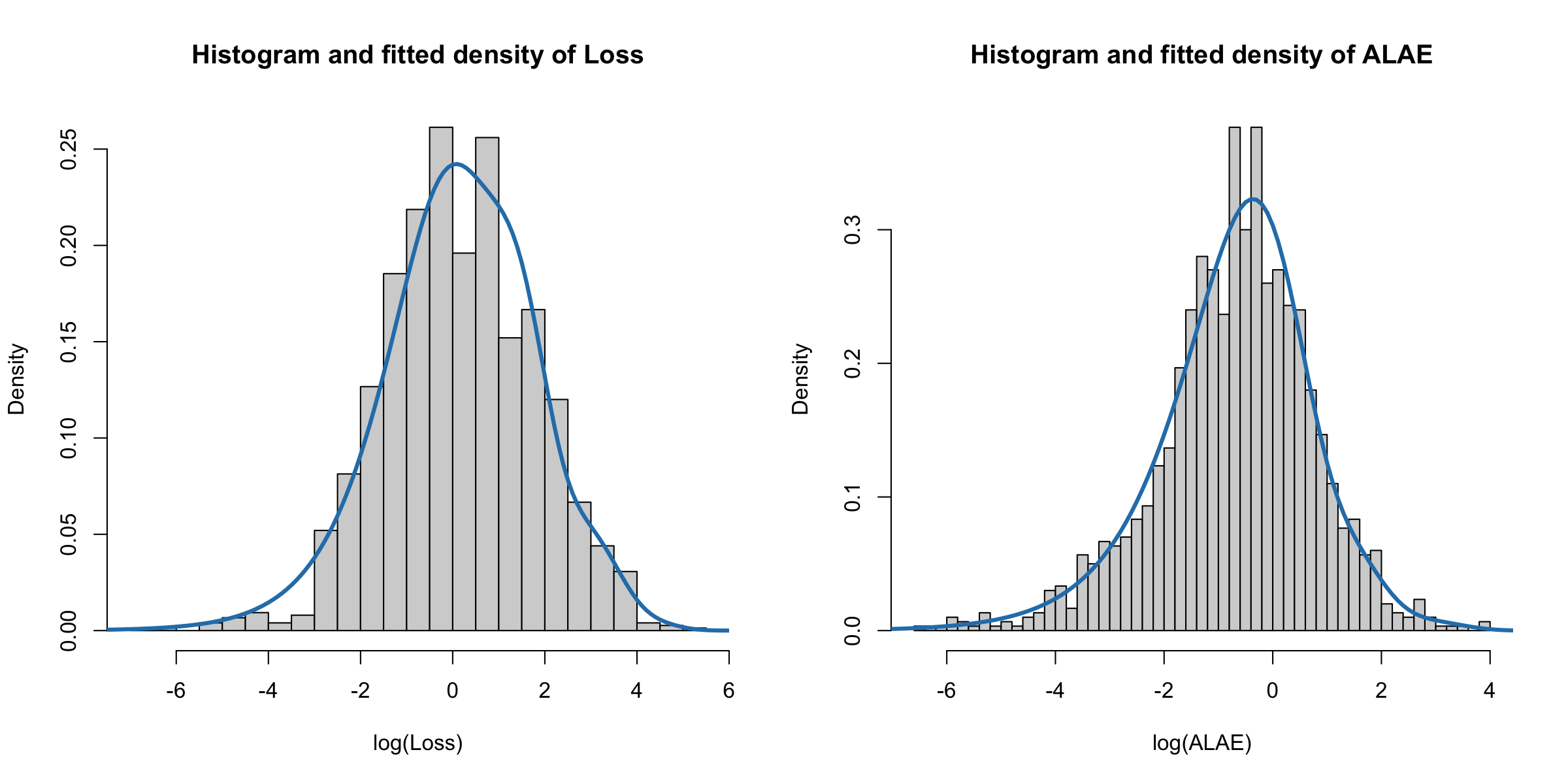}
\caption{Marginal fitted densities for the Loss-ALAE data, together with their respective histograms. Both marginals are log-transformed (after estimation) for visualization purposes.
} \label{histograms}
\end{figure}

\begin{table}[!htbp]
\begin{center}
\begin{tabular}{l c c c c}
\hline
Model& Log Likelihood & Degrees of freedom & AIC & BIC\\
\hline
mPH & $-\textbf{4495.46}$    & $35$   & $9060.921$   & $9246.883$    \\
PH+Gumbel & $-4497.643$     & $39$   & $9073.287$   & $9280.502$    \\
PH+Normal & $-4518.516$     & $39$   & $9115.032$   & $9322.248$    \\
PH+t & $-4510.010$     & $40$   & $9100.020$   & $9312.549$    \\
PH+Frank & $-4530.431$     & $39$   & $9138.862$   & $9346.078$    \\
PH+Clayton & $-4604.792$     & $39$   & $9287.584$   & $9494.800$    \\
PH+Joe & $-4512.921$     & $39$   & $9103.843$   & $9311.058$    \\
PH+Galambos & $-4497.466$     & $39$   & $9072.933$   & $9280.148$    \\
PH+Husler Reiss & $-4502.159$     & $39$   & $9082.319$   & $9289.535$    \\
PH+Tawn & $-4502.549$     & $39$   & $9083.098$   & $9290.314$    \\
\hline
\end{tabular}
\caption{Summary of fitted multivariate models to the Loss-ALAE dataset. All models have general phase-type marginals with $p=4$.}
\label{table:loss_alae_fits}
\end{center}
\end{table}

Finally, we study the empirical versus fitted joint dependence for increasing thresholds. More precisely, we consider joint exceedances over increasing bivariate thresholds. The thresholds for both marginals are chosen according to quantiles of order $\alpha\in(0,1)$. For the fitted counterpart, the correlations for different thresholds are computed from simulations, using $1500$ datapoints, the same number as the sample size. The procedure is the repreated {$5000$} times and averaged. The results are given in Table \ref{table:simstud}, where the tail independence is not yet apparent for the mPH model. Theoretically, for very high thresholds (outside of the data rage), the tail independence will manifest itself.

\begin{table}[!htbp]
\begin{center}
\begin{tabular}{l c c c c c c}
\hline
$\alpha$ & Pearson  & Pearson fitted  & Kendall & Kendall fitted & Spearman & Spearman fitted\\
\hline
$0.1$ & $0.392$ & $0.425$ & $0.292$ & $0.293$ & $0.421$ & $0.426$\\
$0.2$ & $0.379$ & $0.417$ & $0.264$ & $0.277$ & $0.382$ & $0.403$\\
$0.3$ & $0.361$ & $0.410$ & $0.251$ & $0.260$ & $0.363$ & $0.380$\\
$0.4$ & $0.370$ & $0.401$ & $0.237$ & $0.247$ & $0.346$ & $0.362$\\
$0.5$ & $0.353$ & $0.391$ & $0.269$ & $0.242$ & $0.387$ & $0.356$\\
$0.6$ & $0.336$ & $0.377$ & $0.264$ & $0.247$ & $0.383$ & $0.364$\\
$0.7$ & $0.295$ & $0.357$ & $0.248$ & $0.246$ & $0.362$ & $0.360$\\
$0.8$ & $0.235$ & $0.313$ & $0.287$ & $0.200$ & $0.412$ & $0.293$\\
$0.9$ & $0.185$ & $0.270$ & $0.090$ & $0.164$ & $0.135$ & $0.242$\\
\hline
\end{tabular}
\caption{Empirical versus simulated joint exceedance correlations for the Loss-ALAE data.}
\label{table:simstud}
\end{center}
\end{table}

\section{Tail behaviour and extensions}\label{extensions}
Some applications may require tail behaviour in each marginal which is different to that of exponential decay. In particular, life-insurance applications require lighter tails, such as Gompertz-type tails, and third party liability insurance often requires Pareto-type tails. This section outlines two possible solutions, inspired by the univariate case.

\subsection{Inhomogeneous mPH distributions}

Let $ ( J_t^{(k)} )_{t \geq 0}$, $k=1,\dots,d$, be inhomogeneous Mar\-kov pure-jump processes on $\{1, \dots, p, p+1\}$, with states $1,\dots,p$ transient and $p+1$ absorbing. We endow the processes with the same dependence structure as in the homogeneous case, given by \eqref{dependence_def}.

It follows (cf. \cite{albrecher2019inhomogeneous}) that the transition matrix is given by
$$\mat{P}(s,t)=\prod_{s}^{t}(\boldsymbol{I}+\boldsymbol{\Lambda}(u) d u):=\boldsymbol{I}+\sum_{k=1}^{\infty} \int_{s}^{t} \int_{s}^{u_{k}} \cdots \int_{s}^{u_{2}} \mathbf{\Lambda}\left(u_{1}\right) \cdots \mathbf{\Lambda}\left(u_{k}\right) d u_{1} \cdots \mathrm{d} u_{k},$$
with
\begin{align*}
	\mat{\Lambda}(t)= \left( \begin{array}{cc}
		\bfT(t) &  \bft(t) \\
		\0 & 0
	\end{array} \right)\in\mathbb{R}^{(p+1)\times(p+1)}\,, \quad t\geq0\,.
\end{align*}

Then the random variables $X_i = \inf \{ t >  0 : J^{(i)}_t = p+1 \}\,, \: i=1,\dots,d,$ are univariately inhomogeneous phase-type distributed.

\begin{definition}[Inhomogeneous mPH class]
We say that a random vector $X\in\mathbb{R}_+^d$ has a inhomogeneous multivariate phase-type distribution if each component variable $X_i,\:\:i=1,\dots,d$ is the absorption time of distinct inhomogeneous Markov jump process having the structure \eqref{dependence_def} and with $\bfT_i(t)=\lambda_i(t)\bfT_i,$ $i=1,\dots,d.$

\noindent Moreover, we use the notation $$X\sim \mbox{mIPH}(\bfp,\mathcal{T},\mathcal{L}), \quad \mbox{where}\quad \mathcal{T}=\{\bfT_1,\dots,\bfT_d\},\quad \mathcal{L}=\{\lambda_1,\dots,\lambda_d\}.$$
\end{definition}

Writing $$g^{-1}_i(x)=\int_0^x \lambda_i(u)du,\quad i=1,\dots, d,$$
it follows that its density, cumulative distribution function, and tail function are given, for $x\in\mathbb{R}_+^d$, by
\begin{align*}
F_X(x)&=\sum_{j=1}^p\pi_j \prod_{i=1}^d(1-\bfe_j^\mathsf{T}\exp(\bfT_i g^{-1}_i(x_i))\bfe),\\
S_X(x)&=\sum_{j=1}^p\pi_j \prod_{i=1}^d\bfe_j^\mathsf{T}\exp(\bfT_i g^{-1}_i(x_i))\bfe,\\
f_X(x)&=\sum_{j=1}^p\pi_j \prod_{i=1}^d\bfe_j^\mathsf{T}\exp(\bfT_i g^{-1}(x_i))\bft_i\lambda_i(x_i),
\end{align*}
the derivation being analogous to the homogeneous case, and thus omitted. The tail behaviour of each marginal is given by
\begin{align*}
\P(X_i>x_i)&
\sim {c_i [{g_i^{-1}}(x_i)]^{n_i -1} e^{-\chi_2 [{g_i^{-1}}(x_i)]}},\quad i=1,\dots,d,
\end{align*}
where for each $i$, $c_i$ is a positive constant, $-\chi_i$ is the largest real eigenvalue of $\bfT_i$ and $n_i$ is the dimension of the Jordan block associated to $\chi_i$. This can give rise to different tail behaviour for different marginals, which is useful, but uncommon in non-copula specifications.
By construction, it is clear that the marginals are tail independent, and that the copula behaviour is the same as for the mPH class.

\subsection{Fractional mPH distributions}

\subsubsection{A semi-Markov construction}
We first state a univariate construction from \cite{albrecher2019matrix}, which aids us define the multivariate extension. As before, let $E=\{1,2,...,p,p+1\}$ be a state space and denote by
 $\mat{Q}=\{ q_{ij} \}_{i,j\in E}$ the transition matrix of a Markov chain $\{ Y_n \}_{n\in \mathbb{N}}$ on $E$, where the first $p$ states are transient and state $p+1$ is absorbing. In particular we have that the chain $\{ Y_n\}_{n\in\mathbb{N}}$ has a transition matrix given by
\[   \mat{Q} = \begin{pmatrix}
\mat{Q}^1 & \vect{q}^1 \\
\vect{0} & 1 
\end{pmatrix}. \] 
We also assume that $q_{ii}=0$ for all $i\neq p+1$, in order to construct a semi-Markov process. 

Let $\alpha \in (0,1]$ and $\lambda_i>0$. For each state $i=1,...,p$, we define $T^i_n$, $n=1,2,...$ as independent $\mbox{ML}(\alpha,\lambda_i)$--distributed (Mittag-Leffler distributed) random variables, with density given by
\begin{equation}\label{mldens}
  f_{\lambda,\alpha}(x)
=\lambda x^{\alpha-1} E_{\alpha,\alpha}(-\lambda x^\alpha) ,\ \ \ \ \ \lambda>0,\ 0<\alpha\leq 1, \end{equation}
where
\[   E_{\alpha, \beta}(z)=\sum_{k=0}^{\infty} \frac{z^{k}}{\Gamma(\alpha k+\beta)}, \ \ \beta\in\mathbb{R}, \ \alpha >0  \]
is the Mittag--Leffler function (see \cite{gorenflo2014mittag, haubold2011mittag}).

Define $S_0=0$ and
\[ S_n = \sum_{i=1}^n T^{Y_i}_i , \ \ n\geq 1 ,\]
so that the following specification is a semi--Markov process:
\begin{equation}
  J_t = \sum_{n=1}^\infty Y_{n-1} 1\{ S_{n-1}\leq t <S_n  \}  \label{def:X-process_r},\; t\ge 0.
\end{equation}

Several properties of the above construction can be found in \cite{albrecher2019matrix}, including an explicit representation of the transition probabilities in terms of the Mittag-Leffler function, and a fractional version of the Kolmogorov differential equations. Presently we only mention the most relevant property for our purpose:

\begin{theorem}\label{transprobfrac}
Let $\{ J_t\}_{t\geq 0}$ be the semi-Markov process constructed above and let $X=\inf\{t\ge 0:\, J_t= p+1\}$ denote the time until absorption. 
Then $X$ has a PH$_\alpha(\vect{\pi},\mat{T})$ distribution, i.e. has cumulative distribution function given by
$$
F_X(u)=1-\vect{\pi} {E}_{\alpha,1}(\mat{T} u^\alpha) \vect{e} .
$$
\end{theorem}

\subsubsection{The multivariate extension}

Let $ ( J_t^{(k)} )_{t \geq 0}$, $k=1,\dots,d$, be semi-Mar\-kov pure-jump processes as above, all on $\{1, \dots, p, p+1\}$, with states $1,\dots,p$ transient and $p+1$ absorbing. We again endow the processes with the same dependence structure as in the non-fractional case, given by \eqref{dependence_def}.

Then the random variables $X_i = \inf \{ t >  0 : J^{(i)}_t = p+1 \}\,, \: i=1,\dots,d,$ are univariately fractional phase-type distributed.

\begin{definition}[Fractional mPH class]
We say that a random vector $X\in\mathbb{R}_+^d$ has a fractional multivariate phase-type distribution if each component variable $X_i,\:\:i=1,\dots,d$ is the absorption time of a distinct semi-Markov jump process as above, having the dependence structure \eqref{dependence_def}.

\noindent Moreover, we use the notation $$X\sim \mbox{mPH}_\alpha(\bfp,\mathcal{T}), \quad \mbox{where}\quad \mathcal{T}=\{\bfT_1,\dots,\bfT_d\},\quad \alpha (0,1].$$
\end{definition}

The derivation of its density, cumulative distribution function, and tail function are straightforward by combining the non-fractional approach and the formulas in \cite{albrecher2019matrix} for the fractional univariate case, and thus we only state them, for $x\in\mathbb{R}_+^d$:
\begin{align*}
F_X(x)&=\sum_{j=1}^p\pi_j \prod_{i=1}^d(1-\bfe_j^\mathsf{T} {E}_{\alpha,1}(\mat{T} x_i^\alpha) \vect{e}),\\
S_X(x)&=\sum_{j=1}^p\pi_j \prod_{i=1}^d\bfe_j^\mathsf{T} {E}_{\alpha,1}(\mat{T} x_i^\alpha) \vect{e},\\
f_X(x)&=\sum_{j=1}^p\pi_j \prod_{i=1}^dx_i^{\alpha-1}\bfe_j^\mathsf{T}{E}_{\alpha,\alpha}(\mat{T} x_i^\alpha) \vect{t}.
\end{align*}

The tail behaviour of each marginal is regularly varying with index $\alpha$. This follows from the fact that
$$X_i\stackrel{d}{=}\tilde{X}_iS^{(\alpha)}_i,\quad i=1,\dots,d,$$
where $\tilde{X}_i \sim \mbox{PH}(\bfp,\bfT_i)$, and $(S^{(\alpha)}_i)_{i=1,\dots,d}$ is an independent collection of i.i.d. positive stable variables, i.e. with Laplace transform $\exp(-u^\alpha)$. It is not hard to see that even the following relation holds:
\begin{proposition}
Let $X\sim \mbox{PH}_\alpha (\bfp,\mathcal{T})$, and $(S^{(\alpha)}_i)_{i=1,\dots,d}$ as above. Then
$$X=\tilde{X}\bullet S^{(\alpha)},$$
where $\tilde{X}=(\tilde{X}_1,\dots,\tilde{X}_d)\sim \mbox{mPH}(\bfp,\mathcal{T})$, $S^{(\alpha)}=(S^{(\alpha)}_1,\dots,S^{(\alpha)}_d)$, $\tilde{X} {\ci} S^{(\alpha)}$, and $\bullet$ denotes the Schur (component-wise) product of vectors.
\end{proposition}
\begin{proof}
Verification via Laplace transforms or cumulative distribution functions is immediate.
\end{proof}
By construction, the marginals are tail independent, but the copula behaviour in general is different (for $\alpha\neq 1$) than that of the mPH class. See also \cite{abb2020multivariate, multiml} for multivariate fractional PH constructions, with and without possible tail dependence.

\begin{remark}\rm
The adaptations to the estimation methodology are straightforward for mIPH distributions, and rather involved for $\mbox{mPH}_\alpha$ distributions, and thus we omit both cases. It is not a given that an EM algorithm will be faster than naive numerical optimization, especially in the latter case.
\end{remark}
\section{Concluding remarks}\label{sec:conclusion}
We have introduced a new class of multivariate phase-type distributions which is simple to work with probabilistically and has a natural physical interpretation. Common functionals and measures of dependence are explicitly available in terms of matrices, extending and unifying the corresponding formulas for simpler models. Their estimation methodology gains efficiency with respect to more general specifications, and their flexibility was shown theoretically and with some illustrations for a given state-space size.

Several directions of research are promising in conjunction with the mPH class. Exploring the interval-censored and inhomogeneous case will allow the consideration of data with non-exponential tail behaviour, and generally provide more parsimonious estimation. Introduction of covariate information into the initial vector or into the intensity function can further allow for a multivariate regression analysis of claim severity, which can be useful for the simultaneous risk assessment for various lines of business. Finally, a comparison between existing multivariate fractional PH distributions and the mPH$_\alpha$ class could clarify how to introduce tail dependence into mPH models.

\textbf{Acknowledgement.} MB would like to acknowledge financial support from the Swiss National Science Foundation Project 200021\_191984.

\textbf{Declaration} MB declares no conflict of interest related to the current manuscript.

\bibliographystyle{apalike}
\bibliography{dependence.bib}

\begin{thebibliography}{}

\bibitem[Albrecher and Bladt, 2019]{albrecher2019inhomogeneous}
Albrecher, H. and Bladt, M. (2019).
\newblock Inhomogeneous {P}hase-type distributions and heavy tails.
\newblock {\em Journal of Applied Probability}, 56(4):1044--1064.

\bibitem[Albrecher et~al., 2020a]{albrecher2019matrix}
Albrecher, H., Bladt, M., and Bladt, M. (2020a).
\newblock Matrix {M}ittag--{L}effler distributions and modeling heavy-tailed
  risks.
\newblock {\em Extremes}, 23(3):425--450.

\bibitem[Albrecher et~al., 2020b]{abb2020multivariate}
Albrecher, H., Bladt, M., and Bladt, M. (2020b).
\newblock Multivariate fractional phase--type distributions.
\newblock {\em Fractional Calculus and Applied Analysis}, 23(5):1431--1451.

\bibitem[Albrecher et~al., 2020c]{multiml}
Albrecher, H., Bladt, M., and Bladt, M. (2020c).
\newblock Multivariate matrix {M}ittag--{L}effler distributions.
\newblock {\em Annals of the Institute of Statistical Mathematics}, pages
  1--26.

\bibitem[Albrecher et~al., 2021]{albrecher2021continuous}
Albrecher, H., Bladt, M., Bladt, M., and Yslas, J. (2021).
\newblock Continuous scaled phase-type distributions.
\newblock {\em Stochastic Models}, pages 1--30.

\bibitem[Asmussen et~al., 1996]{asmussen1996fitting}
Asmussen, S., Nerman, O., and Olsson, M. (1996).
\newblock Fitting phase-type distributions via the {EM} algorithm.
\newblock {\em Scandinavian {J}ournal of {S}tatistics}, pages 419--441.

\bibitem[Assaf et~al., 1984]{assaf1984multivariate}
Assaf, D., Langberg, N.~A., Savits, T.~H., and Shaked, M. (1984).
\newblock Multivariate phase-type distributions.
\newblock {\em Operations Research}, 32(3):688--702.

\bibitem[Bladt and Nielsen, 2017]{Bladt2017}
Bladt, M. and Nielsen, B.~F. (2017).
\newblock {\em Matrix-Exponential Distributions in Applied Probability}.
\newblock Springer.

\bibitem[Frees and Valdez, 1998]{frees1998understanding}
Frees, E.~W. and Valdez, E.~A. (1998).
\newblock Understanding relationships using copulas.
\newblock {\em North American actuarial journal}, 2(1):1--25.

\bibitem[Furman et~al., 2021]{furman2021}
Furman, E., Kye, Y., and Su, J. (2021).
\newblock Multiplicative background risk models: Setting a course for the
  idiosyncratic risk factors distributed phase-type.
\newblock {\em Insurance: Mathematics and Economics}, 96:153--167.

\bibitem[Gorenflo et~al., 2014]{gorenflo2014mittag}
Gorenflo, R., Kilbas, A.~A., Mainardi, F., and Rogosin, S.~V. (2014).
\newblock {\em Mittag-Leffler functions, related topics and applications}.
\newblock Springer, Berlin.

\bibitem[Haubold et~al., 2011]{haubold2011mittag}
Haubold, H.~J., Mathai, A.~M., and Saxena, R.~K. (2011).
\newblock Mittag-{L}effler functions and their applications.
\newblock {\em Journal of Applied Mathematics}, 2011.
\newblock Article ID 298628, 51 pages.

\bibitem[Joe, 1997]{joe1997multivariate}
Joe, H. (1997).
\newblock {\em Multivariate models and multivariate dependence concepts}.
\newblock CRC press.

\bibitem[Joe, 2014]{joe2014dependence}
Joe, H. (2014).
\newblock {\em Dependence modeling with copulas}.
\newblock CRC press.

\bibitem[Johnson and Taaffe, 1988]{johnson1988denseness}
Johnson, M.~A. and Taaffe, M.~R. (1988).
\newblock The denseness of phase distributions.

\bibitem[Kulkarni, 1989]{kulkarni1989new}
Kulkarni, V.~G. (1989).
\newblock A new class of multivariate phase type distributions.
\newblock {\em Operations Research}, 37(1):151--158.

\bibitem[Lee and Lin, 2012]{lee2012modeling}
Lee, S.~C. and Lin, X.~S. (2012).
\newblock Modeling dependent risks with multivariate {E}rlang mixtures.
\newblock {\em ASTIN Bulletin: The Journal of the IAA}, 42(1):153--180.

\bibitem[Mikosch, 2006]{mikosch2006copulas}
Mikosch, T. (2006).
\newblock Copulas: Tales and facts.
\newblock {\em Extremes}, 9(1):3--20.

\bibitem[Neuts, 1975]{neuts75}
Neuts, M.~F. (1975).
\newblock {Probability distributions of {P}hase type.}
\newblock In {\em Liber Amicorum Professor Emeritus H. Florin}, pages 173--206.
  Department of Mathematics, University of Louvian, Belgium.

\bibitem[Van~Loan, 1978]{van1978computing}
Van~Loan, C. (1978).
\newblock Computing integrals involving the matrix exponential.
\newblock {\em IEEE transactions on automatic control}, 23(3):395--404.

\bibitem[Willmot and Woo, 2015]{willmot2015some}
Willmot, G.~E. and Woo, J.-K. (2015).
\newblock On some properties of a class of multivariate {E}rlang mixtures with
  insurance applications.
\newblock {\em ASTIN Bulletin: The Journal of the IAA}, 45(1):151--173.

\end{thebibliography}

\end{document}